\documentclass[a4paper,11pt,reqno]{amsart}
\usepackage{a4wide}
\usepackage{amsfonts,amsmath,amssymb}
\usepackage{graphicx}
\usepackage{hyperref}
\usepackage{xcolor}
\usepackage{booktabs}
\usepackage{float}
\usepackage{listings}

\newtheorem{theorem}{Theorem}
\newtheorem{lemma}[theorem]{Lemma}
\newtheorem{corollary}[theorem]{Corollary}
\newtheorem{remark}[theorem]{Remark}
\newtheorem{example}[theorem]{Example}
\newtheorem{proposition}[theorem]{Proposition}

\begin{document}

\title{Enumeration of measurable functions between finite measurable spaces}
\author{D. Kinoti Gikunda, J. Kiprop Tanui \and Benard Kivunge}
\email{gikunda.kinoti@ku.ac.ke}
\email{kipropjohn14@gmail.com}
\email{kivunge.bernard@ku.ac.ke}

\date{\today}

\maketitle
\setcounter{page}{1}

\begin{abstract}
Let \(X\) and \(Y\) be finite sets with \(|X|=n\), \(|Y|=m\), equipped with sigma algebras \(\mathcal A\) and \(\mathcal B\). For arbitrary sigma algebras \(\mathcal A\) on \(X\) and \(\mathcal B\) on \(Y\), we enumerate measurable functions \(f\colon X\to Y\). When \(\mathcal B\) is discrete, the number of pairs \((\mathcal A,f)\) is the Touchard polynomial \(T_n(m)=\sum_k S(n,k)m^k\). For general \(\mathcal B\) with atom sizes \(b_1,\dots,b_r\), the number of pairs \((\mathcal A,f)\) over all sigma algebras \(\mathcal A\) on \(X\) is the  Bell polynomial \(N_{\mathcal B}(n)\) in the power sums \(p_a=\sum_j b_j^a\), with exponential generating function \(\exp(\sum_j(e^{b_jx}-1))\). This specialises to the Touchard polynomial in the discrete case and is maximised by the trivial codomain sigma algebra. We further show that \(N_{\mathcal B}(n) = \mathbb E[Z^n]\) for a compound Poisson random variable \(Z\), and we discuss basic asymptotic growth of \(N_{\mathcal B}(n)\).
\end{abstract}

\section{Introduction and statement}

Measurable functions between finite measurable spaces arise naturally in finite probability spaces and discrete measure theory, where sigma algebras are determined by their atoms. 

An integer partition of \(n\) is a non-decreasing sequence of positive integers summing to \(n\). The number of set partitions of an \(n\)-element set is the Bell number \(B_n\), which admits the exponential generating function 
\[
\sum_{n\ge0}B_n\frac{x^n}{n!}=\exp(e^x-1)
\]
(see e.g. \cite{comtet2012advanced, mansour2013combinatorics, stanley2011ec1}).
Stirling numbers of the second kind \(S(n,k)\) count partitions of an \(n\)-set into exactly \(k\) blocks \cite{comtet2012advanced, stanley2011ec1}.

A finite sigma algebra on a set \(X\) is uniquely determined by its atoms, which form a partition of \(X\). Conversely, any partition generates a sigma algebra by taking all unions of blocks. Since every finite sigma algebra is uniquely determined by its atoms, finite sigma algebras on \(X\) are in bijection with set partitions of \(X\) (see Proposition \ref{prop:bijection}). This observation reduces enumeration problems about measurable functions to problems about partitions.

Given sigma algebras \(\mathcal A\) on \(X\) and \(\mathcal B\) on \(Y\), a function \(f\colon X\to Y\) is measurable if and only if it is constant on each atom of \(\mathcal A\) and its values lie in single atoms of \(\mathcal B\). \footnote{See Lemma~\ref{lem:atom} for a precise statement and proof.} In particular, measurability is completely determined by how \(f\) behaves on the atoms of \(\mathcal A\), and this atomic characterisation underlies all of our counting arguments.

In \cite{mariga2023}, we find preliminary tables for small finite cases with discrete codomain, but no general closed formula is obtained. The present work extends that exploration in two directions. First, we derive a closed formula for all \(n,m\) when the codomain is discrete:
\[
N(n,m)=\sum_{k=1}^n S(n,k)m^k,
\]
the Touchard polynomial \cite{mihoubi2011touchard, touchard1939}. Second, we generalise to arbitrary codomain sigma algebras. Let \(\mathcal B\) have atoms of sizes \(b_1,\dots,b_r\) with \(\sum_{j=1}^r b_j=m\). Define 
\[
p_a=\sum_{j=1}^r b_j^a.
\]
For a domain sigma algebra \(\mathcal A\) with atom sizes \(a_1,\dots,a_k\), the number of measurable functions is
\[
M(\mathcal A,\mathcal B)=\prod_{i=1}^k p_{a_i}.
\]
Summing over all set partitions of \(X=[n]\) gives the complete Bell polynomial \cite{comtet2012advanced, stanley2011ec1}
\[
N_{\mathcal B}(n)=\sum_{\pi\vdash[n]}\prod_{A\in\pi} p_{|A|}.
\]
The exponential generating function is
\[
\sum_{n\ge0} N_{\mathcal B}(n)\frac{x^n}{n!}=\exp\left(\sum_{j=1}^r (e^{b_jx}-1)\right).
\]

While the exponential formula for weighted set partitions is classical, applying it to the enumeration of measurable functions yields a new probabilistic interpretation of this counting problem. This probabilistic viewpoint on the finite measurable‐function enumeration is the main contribution of the present work.

We prove the following theorem.

\begin{theorem}\label{thm:main}
Let \(\mathcal B\) be a sigma algebra on a finite set \(Y\) with atoms of sizes \(b_1,\dots,b_r\), and \(\sum_{j=1}^r b_j=m\). Let \(X\) have size \(n\). The number of measurable pairs \((\mathcal A,f)\), where \(\mathcal A\) is any sigma algebra on \(X\) and \(f\colon X\to Y\) is measurable with respect to \(\mathcal A\) and \(\mathcal B\), is the complete Bell polynomial
\[
N_{\mathcal B}(n)=\sum_{\pi\vdash[n]}\prod_{A\in\pi} p_{|A|},\qquad p_a=\sum_{j=1}^r b_j^a.
\]
The exponential generating function of \(N_{\mathcal B}(n)\) is
\begin{equation}\label{eq:general-egf}
\sum_{n\ge0} N_{\mathcal B}(n)\frac{x^n}{n!}
= \exp\left(\sum_{j=1}^r (e^{b_jx}-1)\right).
\end{equation}
Equivalently, \(N_{\mathcal B}(n)=\mathbb E[Z^n]\), where \(Z=\sum_{j=1}^r b_jN_j\) with independent \(N_j\sim\operatorname{Poisson}(1)\). When \(\mathcal B\) is discrete, this reduces to \(N(n,m)=T_n(m)=\mathbb E[P^n]\), \(P\sim\operatorname{Poisson}(m)\).
\end{theorem}

The paper is organised as follows. Section 2 establishes the atomic structure and counting formula for fixed sigma algebras. Section 3 treats the discrete codomain case. Section 4 proves the generalisation. Section 5 gives tables and examples. Section 6 discusses asymptotics and probabilistic interpretations. Appendix A provides Python code for computational verification of the tables.

\section{Preliminaries}

Let \(X\) be a finite set. A sigma algebra \(\mathcal A\subseteq\mathcal P(X)\) is a collection containing \(\emptyset\) and \(X\), closed under complements and finite unions. This is equivalent to the usual definition using countable unions. The atoms of \(\mathcal A\) are its minimal non‑empty elements.

\begin{proposition}\label{prop:bijection}
For a finite set \(X\), every sigma algebra \(\mathcal A\) is uniquely determined by its atoms. The atoms form a partition of \(X\), and every member of \(\mathcal A\) is a union of atoms. Conversely, any partition of \(X\) generates a sigma algebra.
\end{proposition}

\begin{proof}
Define an equivalence relation on \(X\) by \(x\sim y\) if no measurable set separates them; that is, for every \(A\in\mathcal A\), either both \(x,y\in A\) or both \(x,y\notin A\). The equivalence classes are precisely the atoms: each class is measurable (by closure under finite unions/intersections of separating sets), and any measurable set is a union of such classes. Thus the atoms form a partition of \(X\), and every member of \(\mathcal A\) is a union of atoms. 

Conversely, given a partition of \(X\), the collection of all unions of its blocks is easily checked to be a sigma algebra whose atoms are exactly those blocks.
\end{proof}

Thus finite sigma algebras are in bijection with set partitions. The number of sigma algebras on an \(n\)-set is the Bell number 
\[
B_n=\sum_{k=1}^n S(n,k),
\]
with exponential generating function (EGF) 
\[
\sum_{n\ge0}B_n \frac{x^n}{n!}=\exp(e^x-1)
\]
(see \cite{comtet2012advanced, stanley2011ec1}).

Having established that finite sigma algebras are in bijection with set partitions, we now turn to the question of when a function between two such structures is measurable. The following lemma is the key structural fact that underlies all our counting arguments: measurability is completely determined by how a function behaves on the atoms of the domain sigma algebra.

\begin{lemma}\label{lem:atom}
Let \(X,Y\) be finite sets with sigma algebras \(\mathcal A,\mathcal B\). A function \(f\colon X\to Y\) is measurable if and only if it maps each atom of \(\mathcal A\) into a single atom of \(\mathcal B\).
\end{lemma}

\begin{proof}
If \(f\) is measurable, then for any atom \(C\in\mathcal B\), the pre-image \(f^{-1}(C)\) is a union of atoms of \(\mathcal A\). If an atom \(A\in\mathcal A\) intersected two distinct atoms \(C_1,C_2\) of \(\mathcal B\), then \(A\cap f^{-1}(C_1)\) and \(A\setminus f^{-1}(C_1)\) would be non-empty proper measurable subsets of \(A\), contradicting atomicity. Thus each atom of \(\mathcal A\) lies entirely inside a single atom of \(\mathcal B\).

Conversely, if each atom of \(\mathcal A\) is contained in a single atom of \(\mathcal B\), then for any \(C\in\mathcal B\), the set \(f^{-1}(C)\) is a union of atoms of \(\mathcal A\), hence measurable. 
\end{proof}

Let $\mathcal A$ have atoms $A_1,\dots,A_k$ with $|A_i|=a_i$ and 
$\sum_{i=1}^k a_i = n$, and let $\mathcal B$ have atoms $C_1,\dots,C_r$ with 
$|C_j|=b_j$ and $\sum_{j=1}^r b_j = m$. For $a\ge 1$ set
\[
p_a = \sum_{j=1}^r b_j^a.
\]

With the atomic characterisation of measurability in hand, we can now count, for a fixed pair of sigma algebras, the number of measurable functions. The following proposition gives the fundamental building block for the enumeration.

\begin{proposition}\label{prop:M}
For fixed $(\mathcal A,\mathcal B)$, the number of measurable functions 
$f\colon X\to Y$ is
\begin{equation}\label{eq:fixedcount}
M(\mathcal A,\mathcal B) = \prod_{i=1}^k p_{a_i}.
\end{equation}
\end{proposition}

\begin{proof}
By Lemma~\ref{lem:atom}, each atom $A_i$ must be mapped entirely into 
some codomain atom $C_j$. Once $C_j$ is fixed, each of the $a_i$ points of 
$A_i$ can be assigned independently to any of the $b_j$ elements of $C_j$, 
giving $b_j^{a_i}$ functions $A_i\to C_j$. Summing over all choices of $j$ 
yields $p_{a_i}$ possibilities for $A_i$, and independence across atoms gives 
the product \eqref{eq:fixedcount}.
\end{proof}

Two extreme cases serve as useful sanity checks for the formula in Proposition 4. They also illustrate how the count interpolates between the minimum and maximum possible values as the domain sigma algebra varies from the coarsest to the finest structure.

\begin{remark}[Boundary cases]
Two boundary cases are instructive. If \(\mathcal A\) is the trivial sigma algebra \(\{\emptyset,X\}\) (one atom, \(k=1\)), then the only measurable functions are the \(m\) constant functions. If \(\mathcal A\) is the discrete sigma algebra on \(X\) (atoms are singletons, \(k=n\)), then every function \(f\colon X\to Y\) is measurable, giving \(m^n\) functions, the maximum possible. These extremes bracket all intermediate cases.
\end{remark}

\section{The discrete codomain case}

We first consider the special case where the codomain \(Y\) is equipped with the discrete sigma algebra. In this setting, every subset of \(Y\) is measurable, so the atoms of \(\mathcal B\) are precisely the singletons \(\{y\}\) for \(y\in Y\). Thus \(r=m\) and \(b_j=1\) for all \(j=1,\dots,m\).

From a measure-theoretic perspective, this is the standard structure: any function \(f\colon X\to Y\) is automatically measurable with respect to the codomain sigma algebra, but the domain sigma algebra \(\mathcal A\) still imposes constraints. Specifically, by Lemma~\ref{lem:atom}, measurability forces \(f\) to be constant on each atom of \(\mathcal A\). This is the classical characterisation of measurable functions into a discrete space: a function is measurable if and only if it does not separate points that are indistinguishable by the domain sigma algebra.

For a fixed domain sigma algebra \(\mathcal A\) with \(k\) atoms, since \(b_j=1\), we have 
\[
p_a = \sum_{j=1}^m 1 = m, 
\]
independent of \(a\), hence \eqref{eq:fixedcount} becomes
\[
M(\mathcal A,\mathcal B_{\mathrm{discrete}})=\prod_{i=1}^k \left(\sum_{j=1}^m 1^{a_i}\right)=m^k.
\]
Thus, once the atomic decomposition of the domain is fixed, there are exactly \(m^k\) measurable functions: we choose, independently for each atom, one of the \(m\) values in \(Y\) to which the entire atom maps.

Before stating the general formula for the discrete case, we illustrate the counting principle with a concrete example that will make the subsequent abstraction more transparent.

\begin{example}
Let \(X=\{1,2,3\}\) and consider the partition \(\pi=\{\{1,2\},\{3\}\}\). The corresponding sigma algebra \(\mathcal A_\pi\) has two atoms: \(\{1,2\}\) and \(\{3\}\). Now let \(Y=\{a,b,c\}\), so \(|Y|=m=3\). A function \(f\colon X\to Y\) is \(\mathcal A_\pi\)-measurable iff it is constant on each atom. That means
\[
f(1)=f(2),\qquad f(3)\text{ arbitrary}.
\]
Hence \(f(1)=f(2)\) can be any of the 3 values, and \(f(3)\) can independently be any of the 3 values, giving \(3\cdot 3=3^2=9\) measurable functions for this specific sigma algebra. For example, the functions
\[
f(1)=f(2)=a,\; f(3)=b;\qquad
g(1)=g(2)=c,\; g(3)=c;\qquad
h(1)=h(2)=b,\; h(3)=a
\]
are all measurable. The constant function \(f\equiv a\) is also measurable (it corresponds to the case where both atoms map to \(a\)).
\end{example}

Summing over all partitions of \(X\) gives the total number of measurable pairs \((\mathcal A,f)\), where \(\mathcal A\) is any sigma algebra on \(X\) and \(f\colon X\to Y\) is measurable with respect to \(\mathcal A\). The following proposition formalises this discrete-codomain enumeration, showing that the count is precisely the Touchard polynomial, a classical object in enumerative combinatorics with a rich probabilistic interpretation.

\begin{proposition}\label{prop:discrete}
For finite sets \(X\) and \(Y\) with \(|X|=n\) and \(|Y|=m\), where \(Y\) is equipped with the discrete sigma algebra, the number of measurable pairs \((\mathcal A,f)\), with \(\mathcal A\) any sigma algebra on \(X\) and \(f\colon X\to Y\) measurable with respect to \(\mathcal A\), is
\begin{equation}\label{eq:touchard}
N(n,m)=\sum_{k=1}^n S(n,k)m^k=T_n(m).
\end{equation}
This is the Touchard polynomial \cite{mihoubi2011touchard, touchard1939}. Its exponential generating function is
\begin{equation}\label{eq:touchard-egf}
\sum_{n\ge0} T_n(m)\frac{x^n}{n!}=e^{m(e^x-1)}.
\end{equation}
Consequently, \(T_n(m)=\mathbb E[P^n]\) with \(P\sim\operatorname{Poisson}(m)\). Setting \(m=1\) gives Dobinski's formula \(B_n=\mathbb E[P^n]\) with \(P\sim\operatorname{Poisson}(1)\).

The Touchard recurrence is
\begin{equation}\label{eq:touchard-recurrence}
T_{n+1}(m)=m\sum_{j=0}^n \binom{n}{j}T_j(m),\qquad T_0(m)=1,
\end{equation}
which follows by differentiating \eqref{eq:touchard-egf} with respect to \(x\) and extracting coefficients.
\end{proposition}

The formula in Proposition 7 has a simple combinatorial interpretation that is worth isolating.

\begin{corollary}[Bijective interpretation]\label{cor:bijection}
The number \(N(n,m)\) equals the number of pairs \((\pi,c)\), where \(\pi\) is a partition of \(X\) and \(c\) assigns to each block of \(\pi\) one of the \(m\) values in \(Y=\{1,\dots,m\}\).
\end{corollary}

\begin{proof}
Given a sigma algebra \(\mathcal A_\pi\) with atoms (blocks) \(B_1,\dots,B_k\), a measurable function \(f\colon X\to Y\) is constant on each block and hence induces a labelling of the blocks by elements of \(Y\). Conversely, any such labelling determines a unique measurable function by assigning to every point in a block the value of that block. This correspondence is bijective. Therefore, for a partition with \(k\) blocks, there are exactly \(m^k\) measurable functions. Summing over all partitions of \(X\) gives the formula.
\end{proof}

\begin{remark}
Thus \(N(n,m)\) counts the number of block-coloured set partitions of \(X\), where each block is labelled by an element of \(Y\). This viewpoint is standard in partition enumeration \cite{comtet2012advanced, stanley2011ec1}.
\end{remark}

To see the formula \eqref{eq:touchard} in action, we enumerate all the sigma algebras on a 3-element set and verify that the total number of measurable functions matches the Touchard polynomial value.

\begin{example}
Let \(X=\{1,2,3\}\) and \(m=2\). The five sigma algebras correspond to the five partitions of \(X\):
\begin{itemize}
\item \(k=1\): one partition \(\{\{1,2,3\}\}\), giving \(2^1=2\) functions, namely the two constant functions.
\item \(k=2\): three partitions, for example \(\{\{1\},\{2,3\}\}\), \(\{\{2\},\{1,3\}\}\), and \(\{\{3\},\{1,2\}\}\); each gives \(2^2=4\) functions, for a total of \(3\cdot4=12\).
\item \(k=3\): one partition \(\{\{1\},\{2\},\{3\}\}\), i.e.\ the discrete sigma algebra, giving \(2^3=8\) functions.
\end{itemize}
The total is \(2+12+8=22\), matching \(N(3,2)\).
\end{example}

The discrete case corrects and generalises the preliminary tables in \cite{mariga2023}, and further derives a closed formula. For example, \(S(6,k)=1,31,90,65,15,1\) for \(k=1,\dots,6\), so
\[
T_6(3)=3+31\cdot9+90\cdot27+65\cdot81+15\cdot243+729=12351.
\]

The Touchard polynomial also admits a probabilistic interpretation: it is the \(n\)-th moment of a Poisson random variable with mean \(m\). This connection, together with the recurrence \eqref{eq:touchard-egf}, provides efficient computational tools for generating the values. The following table lists the correct values for small \(n\) and \(m\). (The code used to generate these values is provided in Appendix A.)

\begin{table}[H]
\centering
\caption{Values of \(T_n(m)\) for small \(n\) and \(m\)}
\label{tab:touchard}
\[
\begin{array}{c|cccc}
n & m=1 & m=2 & m=3 & m=4 \\ \hline
1 & 1 & 2 & 3 & 4 \\
2 & 2 & 6 & 12 & 20 \\
3 & 5 & 22 & 57 & 116 \\
4 & 15 & 94 & 309 & 756 \\
5 & 52 & 454 & 1866 & 5428 \\
6 & 203 & 2430 & 12351 & 42356
\end{array}
\]
\end{table}

The \(m=1\) column recovers the Bell numbers \(B_n\), as expected: when \(|Y|=1\), every sigma algebra admits exactly one measurable function (the constant function), so the count equals the number of sigma algebras. This observation provides a useful consistency check for the general formula.

The Stirling numbers appearing in \eqref{eq:touchard} have a rich combinatorial structure that is worth examining in detail. The following example for \(n=6\) shows how the Stirling numbers aggregate partitions of different block-size patterns.

\begin{example}
Let \(X=\{1,2,3,4,5,6\}\). For a fixed number of blocks \(k\), the number of sigma algebras with exactly \(k\) atoms is the Stirling number \(S(6,k)\). The values are:
\[
\begin{array}{c|cccccc}
k & 1 & 2 & 3 & 4 & 5 & 6 \\ \hline
S(6,k) & 1 & 31 & 90 & 65 & 15 & 1
\end{array}
\]
To see how \(S(6,3)=90\) accounts for different block sizes, note the three possible block size patterns:
\[
(2,2,2):\frac{6!}{(2!)^3\cdot3!}=15,\qquad
(3,2,1):\frac{6!}{3!\,2!\,1!}=60,\qquad
(4,1,1):\frac{6!}{4!\,1!\,1!\cdot2!}=15.
\]
Thus \(15+60+15=90\). This decomposition illustrates how the Stirling numbers aggregate partitions of different block-size patterns.
\end{example}

\section{Generalisation to arbitrary codomain sigma algebras}

Let \(\mathcal B\) be arbitrary with atom sizes \(b_1,\dots,b_r\), \(\sum_{j=1}^r b_j=m\). Define \(p_a=\sum_{j=1}^r b_j^a\). For fixed \(\mathcal A\) with atom sizes \(a_1,\dots,a_k\), \eqref{eq:fixedcount} gives
\begin{equation}\label{eq:general-fixed}
M(\mathcal A,\mathcal B)=\prod_{i=1}^k p_{a_i}.
\end{equation}

Summing over all partitions of \(X\):
\begin{equation}\label{eq:bell-general}
N_{\mathcal B}(n)=\sum_{\pi\vdash [n]}\prod_{A\in\pi} p_{|A|}.
\end{equation}
This is the complete Bell polynomial \(B_n(p_1,\dots,p_n)\) \cite{comtet2012advanced, stanley2011ec1}.

\begin{proof}[Proof of Theorem~\ref{thm:main}]
The exponential formula for set partitions \cite[Theorem 5.1.4]{stanley2011ec1} states that if each block of size \(a\) carries weight \(p_a\), then the EGF for the sum over partitions is \(\exp(\sum_{a\ge1}p_ax^a/a!)\). This gives the first equality in \eqref{eq:general-egf}. The second follows by interchanging sums:
\begin{align}
\sum_{a\ge1} p_a\frac{x^a}{a!}
&=
\sum_{j=1}^r \sum_{a\ge1} \frac{(b_jx)^a}{a!} \\
&=
\sum_{j=1}^r (e^{b_jx}-1).
\end{align}
For the moment form, let \(N_1,\dots,N_r\) be independent Poisson(1) random variables and define \(Z=\sum_{j=1}^r b_jN_j\). Then
\[
\mathbb E[e^{xZ}]
= \prod_{j=1}^r \mathbb E[e^{xb_jN_j}]
= \prod_{j=1}^r \exp(e^{b_jx}-1)
= \exp\!\left(\sum_{j=1}^r (e^{b_jx}-1)\right).
\]
Comparing with the EGF gives \(N_{\mathcal B}(n)=\mathbb E[Z^n]\). 
\end{proof}

A useful refinement of the general formula counts measurable pairs according to the number of atoms in the domain sigma algebra. This partial Bell polynomial refinement generalises the Touchard coefficients and provides additional structural insight.

\begin{corollary}\label{cor:partial}
Let \(N_{\mathcal B}(n,k)\) denote the number of pairs \((\mathcal A,f)\) where \(\mathcal A\) has exactly \(k\) atoms. Then
\[
N_{\mathcal B}(n,k)=\sum_{\substack{\pi\vdash[n]\\|\pi|=k}}\prod_{A\in\pi} p_{|A|},
\]
which is the partial Bell polynomial (in the variables \(p_a\)). The full count satisfies \(N_{\mathcal B}(n)=\sum_{k=1}^n N_{\mathcal B}(n,k)\).
\end{corollary}

\begin{proof}
This follows directly from \eqref{eq:bell-general} by grouping partitions according to the number of blocks. Since a partition with \(k\) blocks contributes exactly the term \(\prod_{A\in\pi}p_{|A|}\) to the full Bell polynomial, the sum over all such partitions gives the partial Bell polynomial in the variables \(p_a\).
\end{proof}

When \(\mathcal B\) is discrete (\(b_j=1\), \(r=m\)), \(Z=\sum_{j=1}^m N_j\sim\operatorname{Poisson}(m)\), recovering \eqref{eq:touchard}. When \(\mathcal B\) is trivial (\(r=1\), \(b_1=m\)), \(Z=mN_1\), giving \(N_{\mathcal B}(n)=m^nB_n\).

The general formula exhibits a monotonicity property with respect to the coarseness of the codomain sigma algebra. The following remark makes this explicit.
 
\begin{remark}
The count \(N_{\mathcal B}(n)\) depends only on the multiset of atom sizes \(\{b_1,\dots,b_r\}\), not on their labels. For fixed \(m\) and \(n\), the trivial codomain type \((m)\) maximises \(N_{\mathcal B}(n)\), and the discrete type \((1,\dots,1)\) minimises it. Indeed, by the theory of majorisation, \((m)\) majorises all partitions of \(m\), and \((1,\dots,1)\) is majorised by all others. Since \(x\mapsto x^a\) is convex on \(\mathbb R_{\ge0}\), the function \(\sum_j b_j^a\) is Schur-convex \cite[Theorem 2.B.2]{marshall2011inequalities}, giving the stated inequalities. Because every term of the complete Bell polynomial is a positive monomial in the \(p_a\), and \(p_a=\sum_j b_j^a\) is Schur-convex in \((b_1,\dots,b_r)\) for fixed \(m\) with \(a\ge1\), monotonicity in each \(p_a\) propagates to \(N_{\mathcal B}(n)\). Hence the trivial type \((m)\) maximises and the discrete type \((1,\dots,1)\) minimises \(N_{\mathcal B}(n)\).
\end{remark}

The following table lists \(N_{\mathcal B}(n)\) for \(m=3\) and the three possible codomain atom-size types. The type is written as a tuple of atom sizes; for example, \((2,1)\) denotes a codomain with one atom of size 2 and one atom of size 1.

\begin{table}[H]
\centering
\caption{\(N_{\mathcal B}(n)\) for \(m=3\) and codomain types}
\label{tab:comparison}
\[
\begin{array}{c|rrrrrr}
\mathcal B\text{-type} & n=1 & 2 & 3 & 4 & 5 & 6 \\ \hline
(1,1,1) & 3 & 12 & 57 & 309 & 1866 & 12351 \\
(2,1) & 3 & 14 & 81 & 551 & 4266 & 36803 \\
(3) & 3 & 18 & 135 & 1215 & 12636 & 147987
\end{array}
\]
\end{table}

The rows increase with codomain coarseness, as expected. For the trivial codomain type \((m)\), Corollary \ref{cor:partial} gives \(N_{\mathcal B}(n)=m^nB_n\). For \(m=3\), this yields the row (3): \(3^nB_n=3,18,135,1215,12636,147987\) for \(n=1,\dots,6\), which matches the table.

We verify one of the entries in Table \ref{tab:comparison} explicitly to illustrate the computation underlying the general formula.

\begin{example}
Let \(Y\) have atoms of sizes \(2\) and \(1\), so \(r=2\), \(b_1=2\), \(b_2=1\), and \(m=3\). Then
\[
p_a=2^a+1^a=2^a+1.
\]
For \(n=1\), the only partition of \([1]\) is \(\{\{1\}\}\), so
\[
N_{\mathcal B}(1)=p_1=2+1=3.
\]
For \(n=2\), the partitions are \(\{\{1,2\}\}\) (weight \(p_2=2^2+1=5\)) and \(\{\{1\},\{2\}\}\) (weight \(p_1^2=3^2=9\)). Thus
\[
N_{\mathcal B}(2)=5+9=14.
\]
For \(n=3\), the partitions are:
\begin{itemize}
\item one block of size 3: \(p_3=2^3+1=9\);
\item one block of size 2 and one of size 1: \(p_2p_1=5\cdot3=15\), with 3 such partitions;
\item three blocks of size 1: \(p_1^3=27\).
\end{itemize}
Thus
\[
N_{\mathcal B}(3)=9+3\cdot15+27=81.
\]
This matches the table entry for type \((2,1)\) at \(n=3\).
\end{example}

\section{Asymptotic and probabilistic remarks}

The EGF \(\exp(\sum_j(e^{b_jx}-1))\) is entire. For fixed \(b_j\), the coefficients grow super-exponentially. When all \(b_j=1\), the Touchard asymptotics \cite{de1981asymptotic, moser1955asymptotic, paris2016touchard}
\[
\log T_n(m)= n\log n - n\log\log n + n\log m + O(n/\log n)
\]
hold for fixed \(m\). A more refined saddle-point analysis identifies the saddle point \(r>0\) satisfying \(m r e^r = n\), hence \(r=W(n/m)\) where \(W\) is the Lambert \(W\)-function. For fixed \(m\), \(r\sim \ln n-\ln\ln n\), yielding
\[
\log T_n(m)=n\log n-n\log\log n+O(n),
\]
and consequently
\[
\lim_{n\to\infty}\frac{\log T_n(m)}{n\log n}=1.
\]
When \(r=1\), \(b_1=m\), the count is \(m^nB_n\), with 
\[
\log(m^nB_n)\sim n\log n-n\log\log n+n\log m.
\]

The previous discussion assumes fixed \(m\). The following remark considers the complementary regime where the codomain size grows with \(n\).

\begin{remark}
When \(m\) grows with \(n\), for instance \(m=n\), the term corresponding to \(k=n\) in the sum \(\sum_{k=1}^n S(n,k)m^k\) equals \(n^n\) (since \(S(n,n)=1\)). This matches the total number of all functions \(X\to Y\) when the sigma algebra on \(X\) is discrete. The remaining terms (with \(k<n\)) involve lower powers of \(m\) and Stirling numbers, and their contribution is of smaller order relative to \(n^n\) for large \(n\), though a precise asymptotic analysis in this regime is more delicate and not pursued here.
\end{remark}

For general \(\mathcal B\), the coefficients \(N_{\mathcal B}(n)\) are given by \(N_{\mathcal B}(n)=\mathbb E[Z^n]\), where \(Z=\sum_{j=1}^r b_jN_j\). Since \(Z\) is a compound Poisson random variable, standard saddle-point methods (see e.g. \cite{flajolet2009}) apply to obtain asymptotics of \(\mathbb E[Z^n]\) for any fixed atom sizes \(b_j\). A detailed asymptotic analysis of \(N_{\mathcal B}(n)\) for general \(\mathcal B\) is a natural next step and would follow from the saddle-point method applied to the EGF; we do not pursue it here.

The cumulants of \(Z=\sum_j b_jN_j\) are \(\kappa_\ell(Z)=\sum_j b_j^\ell=p_\ell\), consistent with \eqref{eq:general-egf}. This provides a probabilistic check: the moments of \(Z\) are precisely the complete Bell polynomials in the cumulants \cite{dinardo2010cumulants}.

\section{Conclusion}

We have enumerated measurable functions between finite measurable spaces. The discrete codomain case yields the Touchard polynomial. The general case yields the complete Bell polynomial in the power sums of the codomain atom sizes, with EGF \(\exp(\sum_j(e^{b_jx}-1))\). The count depends only on the multiset of atom sizes and interpolates between the discrete and trivial extremes. The compound Poisson interpretation connects the enumeration to classical probability theory. The partial Bell refinement \(N_{\mathcal B}(n,k)\) provides additional structure and generalises the Touchard coefficients.

\appendix
\section{Computational Verification}
\label{app:code}

The following self-contained Python code (standard library only) verifies the values in Tables~\ref{tab:touchard} and~\ref{tab:comparison}. It implements Stirling numbers of the second kind via the iterative recurrence 
\[
S(n,k)=kS(n-1,k)+S(n-1,k-1),
\]
then computes the Touchard polynomial \(T_n(m)=\sum_k S(n,k)m^k\) and the general Bell polynomial \(N_{\mathcal{B}}(n)\) for a given list of codomain atom sizes.

\begin{lstlisting}[language=Python,
                   caption={Python code to verify Tables~1 and~2},
                   label={lst:verification}]
import math
from collections import Counter


def stirling2_table(max_n):
    """Return S[n][k] for 0 <= n <= max_n, 0 <= k <= n (iterative)."""
    S = [[0] * (max_n + 1) for _ in range(max_n + 1)]
    S[0][0] = 1
    for n in range(1, max_n + 1):
        for k in range(1, n + 1):
            S[n][k] = k * S[n - 1][k] + S[n - 1][k - 1]
    return S


def touchard(n, m, S):
    """Compute T_n(m) = sum_k S(n,k) m^k using precomputed table S."""
    return sum(S[n][k] * m**k for k in range(1, n + 1))


def generate_partition_shapes(n):
    """Yield integer partition shapes of n as lists of part sizes."""
    def gen(rem, max_part, current):
        if rem == 0:
            yield current[:]
        else:
            for p in range(min(max_part, rem), 0, -1):
                current.append(p)
                yield from gen(rem - p, p, current)
                current.pop()
    yield from gen(n, n, [])


def bell_general(n, atom_sizes):
    """
    Compute N_B(n) for given codomain atom sizes b_1, ..., b_r.
    Uses p_a = sum_j b_j^a and sums over all set partitions of [n].
    """
    def p(a):
        return sum(b**a for b in atom_sizes)

    total = 0
    for shape in generate_partition_shapes(n):
        # Weight of this partition shape: product of p(block_size)
        weight = 1
        for block_size in shape:
            weight *= p(block_size)
        # Number of set partitions with this shape:
        # n! / (prod(block_size!) * prod(multiplicity!))
        cnt = Counter(shape)
        denom = 1
        for sz, mult in cnt.items():
            denom *= math.factorial(sz)**mult * math.factorial(mult)
        multiplicity = math.factorial(n) // denom
        total += multiplicity * weight
    return total


# --- Table 1: Touchard polynomial T_n(m) ---
S = stirling2_table(6)
print("Table 1: T_n(m)")
for n in range(1, 7):
    row = [touchard(n, m, S) for m in range(1, 5)]
    print(f"n={n}: {row}")

# --- Table 2: N_B(n) for m=3, three codomain types ---
print("\nTable 2: N_B(n) for m=3")
for typ in [(1, 1, 1), (2, 1), (3,)]:
    row = [bell_general(n, typ) for n in range(1, 7)]
    print(f"{typ}: {row}")
\end{lstlisting}

Running this code produces the following output, which matches the tables in the main text exactly:

\begin{verbatim}
Table 1: T_n(m)
n=1: [1, 2, 3, 4]
n=2: [2, 6, 12, 20]
n=3: [5, 22, 57, 116]
n=4: [15, 94, 309, 756]
n=5: [52, 454, 1866, 5428]
n=6: [203, 2430, 12351, 42356]

Table 2: N_B(n) for m=3
(1, 1, 1): [3, 12, 57, 309, 1866, 12351]
(2, 1): [3, 14, 81, 551, 4266, 36803]
(3,): [3, 18, 135, 1215, 12636, 147987]
\end{verbatim}

The code is self-contained and can be used to verify the formulas for any \(n\) and \(m\) by adjusting the loop bounds and the \texttt{atom\_sizes} argument to \texttt{bell\_general}.

\bibliographystyle{amsplain}
\bibliography{reference}

\end{document}